%% file: main.tex
\documentclass[12pt]{iopart}

\include{frankopts}
\include{franknames}

\begin{document}

\input{title}

\input{introduction}

\input{balancing}

\input{analysis}

\input{noisebehavior}

\input{conclusion}

\section*{References}
\bibliographystyle{amsalpha}
\bibliography{bibliography}

\end{document}

%% file: frankopts.tex
\usepackage{amsmath}

\usepackage{dsfont}
\usepackage{mathrsfs}
\usepackage{lmodern}
\usepackage[T1]{fontenc}
\usepackage[latin1]{inputenc}

\usepackage{iopams}

\usepackage[amsmath,thmmarks]{ntheorem}

\newtheorem{theorem}{Theorem}[section]
\newtheorem{lemma}[theorem]{Lemma}

\newtheorem{corollary}[theorem]{Corollary}
\newtheorem{definition}[theorem]{Definition}

\newtheorem{remark}[theorem]{Remark}

\theoremstyle{nonumberbreak}
\theorembodyfont{\normalfont} %
\theoremsymbol{$\Box$}
\newtheorem{proof}{Proof}

\renewcommand{\(}{\left(}
\renewcommand{\)}{\right)}

\newcommand{\<}{\left<}
\renewcommand{\>}{\right>}

%% file: franknames.tex
\newcommand{\IN}          {\mathds{N}}                                        

\newcommand{\SpX}         {\mathcal{X}}
\newcommand{\SpY}         {\mathcal{Y}}

\newcommand{\VarX}        {x}                                                 
\newcommand{\VarY}        {y}                                                 
\newcommand{\Err}         {\xi}

\newcommand{\Op}          {A}                                                 

\newcommand{\BasX}        {u}                                                 
\newcommand{\BasY}        {v}                                                 

\newcommand{\Fsing}       {\mathfrak{s}}                                            

\newcommand{\E}           {\mathbb{E}}
\newcommand{\PP}          {\mathbb{P}}

\newcommand{\argmin}      {\operatorname*{argmin}}


%% file: title.tex
\title{Applying Lepskij-Balancing in Practice}
\author{Frank Bauer}

\eads{frank.bauer.de@gmail.com}

\date{This version: \today}

\begin{abstract}
In a stochastic noise setting the Lepskij balancing principle for choosing the regularization parameter in the regularization of inverse problems is depending on a parameter $\tau$ which in the currently known proofs is depending on the unknown noise level of the input data. However, in practice this parameter seems to be obsolete.

We will present an explanation for this behavior by using a stochastic model for noise and initial data. Furthermore, we will prove that a small modification of the algorithm also improves the performance of the method, in  both speed and accuracy.
\end{abstract}

\ams{47A52,65J22,60G99,62H12}

\maketitle

%% file: introduction.tex
\section{Introduction}

In the following, we will consider linear inverse problems \cite{Engl/Hanke/Neubauer:1996,Hofmann:1986}
given as an operator equation
\begin{equation}\label{main}
 \Op \VarX = \VarY,
\end{equation}
where $\Op :\SpX\to \SpY$ is a linear, continuous,
compact operator acting between separable real infinite dimensional Hilbert spaces
$\SpX,\SpY$. Without loss of generality we assume that $\Op$ has a trivial null-space $N(\Op) = \{ 0 \}$. $\Op$ does not have a continuous inverse because $\Op$
is compact and $\SpX$ is infinite dimensional, and hence
\eqref{main} is ill-posed.

For the analysis we will need the singular value decomposition of $\Op$. There exist orthonormal bases $(u_k)_{k \in \IN}$ of $\SpX$
and  $(v_k)_{k \in \IN}$ of $\SpY$ and a sequence of positive decreasing singular values $(\Fsing_k)_{k \in \IN}$ such that
\begin{equation}
 \Op x = \sum_{k=1}^\infty \Fsing_k \<x,u_k\> v_k.
 \end{equation}
Moreover, we assume that the data $\VarY$ are noisy, the noise model for $\Err$ will be specified later, in contrast to the classical considerations in a stochastic setting $\xi$ is not necessarily an element of $\SpY$.

\begin{equation}
\VarY^\delta = \Op \VarX + \Err ,\quad\text{$\Err$ noise}.
\end{equation}
In order to counter the ill-posedness, we need to regularize; in this article we will concentrate on the regularization method truncated singular value decomposition (TSVD, also called spectral cut-off regularization) which has some specific features that make proofs considerably easier. The level $n$ at which we truncate is called regularization parameter. The subsampling function $s(\cdot) : \IN \mapsto \IN$ is assumed to be strictly increasing.
\begin{equation}
\Op_n^{-1} \VarY^\delta = \VarX_n^\delta = \sum_{k=1}^{s(n)} \(\< \VarX , \BasX_k \> + \Fsing_k^{-1} \< \xi , \BasY_k \> \)  \BasX_k
\end{equation}
The unknown noise-free regularized solution is defined as
 \begin{equation}
\Op_n^{-1} \VarY = \VarX_n = \sum_{k=1}^{s(n)} \< \VarX , \BasX_k \>  \BasX_k
\end{equation}
The correct choice of the regularization parameter is of major importance for the performance of the method.
In recent times, a number of articles \cite{Goldenshluger/Pereverzev:2000,Mathe/Pereverzev:2003,Bauer/Pereverzev:2005,Mathe/Pereverzev:2006,Haemarik/Palm/Raus:2007,Bauer/Hohage/Munk:2009} have considered the Lepskij Balancing principle \cite{Lepskij:1990} for choosing this parameter in various situations.
For practical applications there are still three open issues:
\begin{itemize}
\item In the case of stochastic noise, one loses, in comparison to the optimal situation, a logarithmic factor; i.e. the proven convergence rate of the error is $O(\delta^H\log(\delta))$ in comparison to an optimal $O(\delta^H)$ where $\xi = \delta \overline{\xi}$ with a normalized $\overline{\xi}$,  $H$ is depending on $x$ and $\xi$. This phenomenon cannot be observed in practical implementations; the question is why?
\item In practical implementations, one can replace some knowledge needed explicitly in the proofs (the size of the regularized error in $\SpX$) with a data-driven approximation without losing performance. Can this be put on a firm mathematical basis?
\item Is there a possibility to improve the speed of the method such that it can compete with others, e.g. the Morozov Discrepancy principle \cite{Engl/Hanke/Neubauer:1996,Morozov:1966}?
\end{itemize}
In order to explain some behavior observed using other parameter choice methods, in practical situations an alternative model for describing the solution and the noise has recently been proven successful \cite{Bauer/Reiss:2008,Bauer/Kindermann:2008}. Using this model, we can answer the questions posed above by slightly modifying Lepskij's algorithm such that we can prove an oracle inequality.

The outline of the article is as follows. First we will cite the definition of the Lepskij Balancing principle. Then we will define our model and calculate the underlying expectations on whose basis we will estimate the probabilities that the balancing principle behaves differently than expected. This will yield the desired oracle inequality.

Using the same methodology, we will show that an estimation based on two measurements is sufficient to obtain the same result, of course with weaker constants.

%% file: balancing.tex
\section{Lepskij Balancing Principle}

The key point in the Lepskij Balancing Principle is the knowledge of the noise behavior, which has different forms for different noise regimes \cite{Goldenshluger/Pereverzev:2000,Mathe/Pereverzev:2003,Bauer/Pereverzev:2005}.
\begin{definition}[Noise Behavior]
If $\xi$ is assumed to be in a deterministic regime (i.e., $\| \xi\|\leq \delta$), then define
\begin{equation}
\varrho(n) := \Fsing_{s(n)}^{-1} \delta \geq \| \Op_n^{-1} \xi \|
\end{equation}
where $\delta$ is the noise level. If $\xi$ is assumed to be stochastic, then define
\begin{equation}
\varrho(n)^2 := \E \| \Op_n^{-1} \xi \|^2.
\end{equation}
Later on we will specify more precisely what we mean by stochastic. In both cases, $\varrho(\cdot)$ is a monotonically increasing function.
\end{definition}
Now we will follow the approach presented in \cite{Bauer/Munk:2007}, which already incorporates the (minor) modifications of the balancing principle to make it fit for practice, in particular, by limiting the number of necessary computations.
\begin{definition}[Special parameters]
There are two special regularization parameters which are important for the later proofs:
\begin{itemize}
\item $n_{opt}$: the optimal regularization parameter, i.e., we have $\|\Op_{n_{opt}}^{-1} \Op \VarX\| \approx \varrho(n_{opt})$. The parameter $n_{opt}$ is generally unknown.
\item $N$: the maximal regularization parameter, i.e., the point where one can be sure that in any case $n_{opt} < N$. Even when one has just a very rough idea of the noise, respectively the noise level $\delta$, this parameter can be estimated rather reliably. (E.g., in the deterministic case: $N = \varrho^{-1}(\delta)$, see \cite{Mathe/Pereverzev:2003}, for a statistical setup \cite{Mathe/Pereverzev:2006}).

    However, assuming the knowledge of such a parameter $N$ is problematic at some point; it is likely that a number of other parameter choice methods would work better if one were able to detect outliers easily.
\end{itemize}
\end{definition}
\begin{definition}[Look-Ahead]
Let $\sigma > 1$. Define the look-ahead function by
\begin{equation*}
 l_{N,\sigma}(n) = \min \{ \min \{m | \varrho(n)^{-1} > \sigma \varrho(m)^{-1} \}, N \}
\end{equation*}
\end{definition}

\begin{definition}[Balancing Functional]
The balancing functional is defined as
\begin{equation*}
b_{N,\sigma}(n) = \max_{n<m\leq l_{N,\sigma}(n)} \left\{ 4^{-1} \| \VarX_n^\delta - \VarX_m^\delta \| \varrho(m)^{-1}  \right\}.
\end{equation*}
The smoothed balancing functional is defined as
\begin{equation}\label{Bn}
B_{N,\sigma}(n) = \max_{n\leq m\leq N} \left\{ b_{N,\sigma}(m) \right\}.
\end{equation}
\end{definition}

\begin{definition}[Balancing Stopping Index]
The balancing stopping index is defined as
\begin{equation}\label{balanceindex}
n_{N,\sigma,\kappa} = \min_{n \leq N} \left\{ B_{N,\sigma}(n) \leq \kappa \right\}.
\end{equation}
If no ambiguities can occur, we will denote $n_{N,\sigma,\kappa}$ by $n_*$
\end{definition}

\begin{remark}
A number of results and facts are known:
\begin{itemize}
\item The classical proofs are for $\sigma =\infty$, i.e. $l_{N,\infty}(n) = N$. However, reducing $\sigma$ just worsens some constants.
\item In the case of deterministic noise, $\kappa =1$. Then it holds \cite{Mathe/Pereverzev:2003}
\begin{equation*}
\|\VarX - \VarX_{n_*}^\delta \| \leq c \( \| \Op_n^{-1} \Op \VarX_{n_{opt}} \| + \varrho(n_{opt}) \)
\end{equation*}
where $c$ is independent of $\VarX$ and $\xi$.
\item In the case of stochastic noise and $\kappa = \varrho(N)$ it holds \cite{Goldenshluger/Pereverzev:2000}
\begin{equation*}
\sqrt{\E \|\VarX - \VarX_{n_*}^\delta \|^2} \leq c \log(\varrho(N)) \( \| \Op_n^{-1} \Op \VarX_{n_{opt}} \| + \varrho(n_{opt}) \)
\end{equation*}
where $c$ independent of $\VarX$ and $\xi$.
\item These results are basically independent of the regularization method, i.e. they also apply to other well known methods like Tikhonov regularization and Landweber iteration \cite{Mathe/Pereverzev:2003}.
\item Similar results hold for non-linear inverse problems in combination with the Iteratively Regularized Gau\ss-Newton Method (IRGNM) \cite{Bauer/Hohage/Munk:2009}.
\end{itemize}
\end{remark}

%% file: analysis.tex
\section{A Closer Analysis}

In order to analyze the behavior of the methods in practice, we will now use the
Bayesian model introduced in \cite{Bauer/Reiss:2008}.
\begin{equation*}
\<\VarX, \BasX_k\> \sim \mathcal{N}(0,(\eta k^{-\gamma})^2)
\end{equation*}
\begin{equation*}
\Fsing_k = k^{-\lambda}
\end{equation*}
\begin{equation*}
\<\xi, \BasY_k\> \sim \mathcal{N}(0,(\delta k^{\varepsilon})^2)
\end{equation*}
where
$\gamma > 1/2$ ,$\lambda > 0$, $\lambda > -\varepsilon$ and all Gaussian random variables are independent and identically distributed (iid). All expectations $\E$ should now be interpreted as joint expectations of $\VarX$ and $\xi$.

\subsection{Spectral Cut-Off Regularization}
\begin{definition}[Subsampling]
Let $\omega_0 > 1$, $\omega > 1$ and $\omega_0 \omega > \omega_0 +1$. We choose the following subsampling for obtaining the regularization parameter:
\begin{equation*}
s(n)= \lceil\omega_0 \omega^n\rceil
\end{equation*}
\end{definition}
\begin{remark}
Due to $\omega_0 \omega > \omega_0 +1$ it always holds $s(n+1)>s(n)$. Furthermore we have
\begin{equation*}
\omega_0 \omega^n \leq s(n) \leq \frac{\omega_0+1}{\omega_0} \omega_0 \omega^n
\end{equation*}
\end{remark}
Basic calculus using upper and lower sums to approximate an integral yields
\begin{lemma}
Let $m/\omega \geq n \geq \omega_0$. If  $\kappa > 1$ then
\begin{equation*}
\(1  - \omega^{-\kappa+1} \)\frac{1}{\kappa -1}  n^{-\kappa+1} <  \sum_{k=n}^{m-1} k ^{-\kappa} < \(\frac{\omega_0-1}{\omega_0}\)^{-\kappa+1}  \frac{1}{\kappa -1}  n^{-\kappa+1}.
\end{equation*}
If $\kappa \geq 0$ then
\begin{equation*}
\(\frac{\omega_0-1}{\omega_0}\)^{\kappa+1}\(1 - \omega^{-\kappa-1}\) \frac{1}{1+\kappa}  m^{\kappa+1}  < \sum_{k=n}^{m-1} k ^{\kappa} < \frac{1}{1+\kappa} m^{\kappa+1}.
\end{equation*}
\end{lemma}
\begin{corollary}[Adjacent Difference]
Let $0\leq n < m$. Then it holds
\begin{align}
& \nonumber c_1\(\frac{\eta^2 \omega_0^{-2\gamma +1}}{ 2\gamma -1} \omega^{n(-2\gamma +1)} +
 \frac{\delta^2 \omega_0^{1+2\lambda+2\varepsilon}}{1+2\lambda+2\varepsilon}  \omega^{m(2\lambda+2\varepsilon+1)}\)
\\&\qquad\qquad \leq \E \|\VarX_m^\delta - \VarX_n^\delta\|^2 \leq c_2\(\frac{\eta^2 \omega_0^{-2\gamma +1}}{ 2\gamma -1} \omega^{n(-2\gamma +1)} +
 \frac{\delta^2 \omega_0^{1+2\lambda+2\varepsilon}}{1+2\lambda+2\varepsilon}  \omega^{m(2\lambda+2\varepsilon+1)}\)\label{ineq1}
\end{align}
with
\begin{align*}
c_1 =& \min\left\{\(\frac{\omega_0+1}{\omega_0}\)^{-2\gamma+1}(1 - \omega^{-2\gamma+1}) , \(\frac{\omega_0-1}{\omega_0}\)^{2\lambda+2\varepsilon+1}\(1 - \omega^{-2\lambda-2\varepsilon-1}\) \right\}\\
c_2 =& \max\left\{  \(\frac{\omega_0-1}{\omega_0}\)^{-2\gamma+1}  , \(\frac{\omega_0+1}{\omega_0}\)^{1+2\lambda+2\varepsilon} \right\}
\end{align*}
\end{corollary}
\begin{proof}
It holds
\begin{equation*}
\VarX_m^\delta - \VarX_n^\delta = \sum_{k=s(n)}^{s(m)-1} \(\< \VarX , \BasX_k \> + \sigma_k^{-1} \< \xi , \BasY_k \> \)  \BasX_k
\end{equation*}
and hence
\begin{align*}
\E \|\VarX_m^\delta - \VarX_n^\delta\|^2 &= \sum_{k=s(n)}^{s(m)-1} \eta^2 k^{-2\gamma} + \delta^2 k^{2\lambda+2\varepsilon}
\end{align*}
and hence
\begin{align*}
&\(\frac{\omega_0+1}{\omega_0}\)^{-2\gamma+1}(1 - \omega^{-2\gamma+1}) \frac{\eta^2 \omega_0^{-2\gamma +1}}{ 2\gamma -1} \omega^{n(-2\gamma +1)}
\\& \qquad+ \(\frac{\omega_0-1}{\omega_0}\)^{2\lambda+2\varepsilon+1}\(1 - \omega^{-2\lambda-2\varepsilon-1}\) \frac{\delta^2 \omega_0^{1+2\lambda+2\varepsilon}}{1+2\lambda+2\varepsilon}  \omega^{m(2\lambda+2\varepsilon+1)}
\\& \qquad\qquad\leq \E \|\VarX_m^\delta - \VarX_n^\delta\|^2 \leq \(\frac{\omega_0-1}{\omega_0}\)^{-2\gamma+1} \frac{\eta^2 \omega_0^{-2\gamma +1}}{ 2\gamma -1} \omega^{n(-2\gamma +1)}
\\& \qquad\qquad\qquad+ \(\frac{\omega_0+1}{\omega_0}\)^{1+2\lambda+2\varepsilon}\frac{\delta^2 \omega_0^{1+2\lambda+2\varepsilon}}{1+2\lambda+2\varepsilon}  \omega^{m(2\lambda+2\varepsilon+1)}
\end{align*}
which yields the proposition.
\end{proof}
\begin{corollary}[Propagated Noise]
Let $0\leq n < m$. Then it holds
\begin{align}
&  c_3
 \frac{\delta^2 \omega_0^{1+2\lambda+2\varepsilon}}{1+2\lambda+2\varepsilon}  \omega^{m(2\lambda+2\varepsilon+1)}
 \leq \varrho(m)^2 \leq c_4
 \frac{\delta^2 \omega_0^{1+2\lambda+2\varepsilon}}{1+2\lambda+2\varepsilon}  \omega^{m(2\lambda+2\varepsilon+1)}\label{ineq2}
\end{align}
with
\begin{align*}
c_3 =& \(\frac{\omega_0-1}{\omega_0}\)^{2\lambda+2\varepsilon+1}\(1 - \omega^{-2\lambda-2\varepsilon-1}\) \\
c_4 =&  \(\frac{\omega_0+1}{\omega_0}\)^{1+2\lambda+2\varepsilon}
\end{align*}
\end{corollary}
\begin{proof}
Using
\begin{align*}
\varrho(m)^2 = \E \|\VarX_m^\delta - \VarX_m\|^2 &= \sum_{k=1}^{s(m)-1} \delta^2 k^{2\lambda+2\varepsilon} \\
\end{align*}
we can proceed as beforehand.
\end{proof}
\begin{corollary}[Regularization Error]
Let $0\leq n < m$. Then it holds
\begin{align}
& \nonumber c_5\(\frac{\eta^2 \omega_0^{-2\gamma +1}}{ 2\gamma -1} \omega^{n(-2\gamma +1)} +
 \frac{\delta^2 \omega_0^{1+2\lambda+2\varepsilon}}{1+2\lambda+2\varepsilon}  \omega^{n(2\lambda+2\varepsilon+1)}\)
\\&\qquad\qquad \leq \E \|\VarX_n^\delta - \VarX\|^2 \leq c_6\(\frac{\eta^2 \omega_0^{-2\gamma +1}}{ 2\gamma -1} \omega^{n(-2\gamma +1)} +
 \frac{\delta^2 \omega_0^{1+2\lambda+2\varepsilon}}{1+2\lambda+2\varepsilon}  \omega^{n(2\lambda+2\varepsilon+1)}\)\label{ineq3}
\end{align}
with
\begin{align*}
c_5 =& \min\left\{\(\frac{\omega_0+1}{\omega_0}\)^{-2\gamma+1}(1 - \omega^{-2\gamma+1}) , \(\frac{\omega_0-1}{\omega_0}\)^{2\lambda+2\varepsilon+1}\(1 - \omega^{-2\lambda-2\varepsilon-1}\) \right\}\\
c_6 =& \max\left\{  \(\frac{\omega_0-1}{\omega_0}\)^{-2\gamma+1}  , \(\frac{\omega_0+1}{\omega_0}\)^{1+2\lambda+2\varepsilon} \right\}
\end{align*}
\end{corollary}
\begin{proof}
Using
\begin{align*}
\E \|\VarX_n^\delta - \VarX\|^2 &= \sum_{k=s(n)}^{\infty} \eta^2 k^{-2\gamma} + \sum_{k=1}^{s(n)-1} \delta^2 k^{2\lambda+2\varepsilon}
\end{align*}
we can proceed as beforehand.
\end{proof}
\begin{remark}
Obviously it holds
$c_1=c_5 < c_3 < 1 < c_4 < c_2=c_6$ where we can get as close to $1$ as we want, as long as for fixed $\omega$ the constant $\omega_0$ is big enough.

Although this constant $\omega_0$ will have large influence in the latter proofs we cannot observe in practice \cite{Bauer/Lukas:2010} any major influence; $\omega_0 =3$ seems to be sufficient in most situations even when $\omega$ is rather close to $1$.

 As $\omega_0$ is independent of the noise level $\delta$ we have that at least all proofs hold asymptotically. An explication for the insensitivity in practice  towards $\gamma$ and the other parameters might be that our inequalities to handle the probabilities are too conservative.
\end{remark}
Now we can approximately determine the expected minimal point for $\E \|\VarX_n^\delta - \VarX\|^2$:
\begin{equation*}
\E \|\VarX_n^\delta - \VarX_n\|^2 = \E\|\VarX_n^0 - \VarX\|^2
\end{equation*}
which yields
\begin{equation}
\frac{\eta^2 \omega_0^{-2\gamma +1}}{ 2\gamma -1}  \omega^{n_{opt}(-2\gamma +1)} = \frac{\delta^2 \omega_0^{2\lambda+2\varepsilon +1}}{2\lambda+2\varepsilon +1} \omega^{n_{opt}(2\lambda+2\varepsilon +1)}\label{nopt}
\end{equation}
i.e.,
\begin{equation*}
\frac{\eta^2}{ 2\gamma -1} s(n_{opt})^{-2\gamma +1} = \frac{\delta^2}{2\lambda+2\varepsilon +1} s(n_{opt})^{2\lambda+2\varepsilon +1}
\end{equation*}
and hence
\begin{equation*}
s(n_{opt}) = \(\frac{\eta^2}{\delta^2} \frac{2\gamma -1}{2\lambda+2\varepsilon +1}\)^{1/(2\lambda+2\varepsilon+2\gamma)}
\end{equation*}
respectively
\begin{equation*}
n_{opt} = \log \(\(\frac{\eta^2}{\delta^2} \frac{2\gamma -1}{2\lambda+2\varepsilon +1}\)^{1/(2\lambda+2\varepsilon+2\gamma)}\omega_0^{-1} \) / \log \omega
\end{equation*}
Obviously $n_{opt}$ does not need to exist if $\omega_0$ is getting too big. However, for the rest of the article we will assume the existence of $n_{opt}$ as there exists (depending on $\omega_0$) a $\delta_0$ such that $n_{opt}$ exists for any $\delta<\delta_0$.

Additionally, it holds
\begin{equation*}
l_{N,\sigma}(n) = n + K
\end{equation*}
for some fixed $K \approx \log(\sigma)/\log(\omega)$.
Furthermore, we have a lemma which was proven in \cite{Bauer/Reiss:2008}.
\begin{lemma}\label{ProbLemma}
Let $Z=\sum_{k=1}^\infty\alpha_k^2\zeta_k^2$ with $\sum_{k=1}^\infty\alpha_k^2=1$ and
$\zeta_k\sim N(0,1)$ iid. Assume that $\max_k \alpha_k >0$. Then
\begin{align} &\forall\,z\in (0,1):\;\PP(Z\le z)\le \exp\(\frac{1-z+\log(z)}{2\max_k\alpha_k^2} \) \le (e z)^{\frac{1}{{2\max_k\alpha_k^2}}} \label{ineq:gauss1} \\
&\forall\,z>0:\;\PP(Z\ge z)\le \sqrt{2}e^{-z/4}. \label{ineq:gauss2}
\end{align}
\end{lemma}
Now we will evaluate the probabilities.
\begin{lemma}\label{lemma:lowerbound}
Assume that $n_{opt} < n$ and that $\omega_0$ is big enough such that
\begin{equation}
\frac{c_3}{c_6}  \geq \frac{1}{2}. \label{hi:1}
\end{equation}
Then it holds that
\begin{equation*}
\PP \{ b_{N,\sigma}(n) > \tau \} \leq K \sqrt{2}e^{-\tau^2}
\end{equation*}
and
\begin{equation*}
\PP \{ B_{N,\sigma}(n) > \tau \} \leq K \frac{\log \frac{\delta}{\omega_0}}{-\lambda
\log \omega} \sqrt{2}e^{-\tau^2}.
\end{equation*}
\end{lemma}
\begin{proof}
It holds due to \eref{ineq1}, \eref{ineq2}
and \eref{ineq:gauss2}
\begin{align*}
\PP \{ b_{N,\sigma}(n) > \tau \}
&\leq \sum_{1\leq k \leq K}\PP \left\{ 4^{-1} \| \VarX_n^\delta - \VarX_{n+k}^\delta \| \varrho(n+k)^{-1} > \tau \right\}\\
&\leq K \max_{1\leq k \leq K}\PP \left\{ 4^{-1} \| \VarX_n^\delta - \VarX_{n+k}^\delta \| \varrho(n+k)^{-1} > \tau \right\}\\
&= K \max_{1\leq k \leq K}\PP \left\{ \frac{\| \VarX_n^\delta - \VarX_{n+k}^\delta \|^2}{\E\| \VarX_n^\delta - \VarX_{n+k}^\delta \|^2} > 16 \tau^2 \frac{\varrho(n+k)^2}{\E\| \VarX_n^\delta - \VarX_{n+k}^\delta \|^2} \right\}\\
&\overset{\eref{ineq1}\eref{ineq2}}{\leq} K \max_{1\leq k \leq K}\PP \left\{ \frac{\| \VarX_n^\delta - \VarX_{n+k}^\delta \|^2}{\E\| \VarX_n^\delta - \VarX_{n+k}^\delta \|^2} >
16 \tau^2 \frac{c_3  \frac{\delta^2 \omega_0^{2\lambda+2\varepsilon +1}}{2 \lambda+2\varepsilon +1} \omega^{(n+k)(2\lambda+2\varepsilon +1)}}{2 c_6\frac{\delta^2 \omega_0^{2\lambda+2\varepsilon +1}}{2\lambda+2\varepsilon +1} \omega^{(n+k)(2\lambda+2\varepsilon +1)}} \right\}\\
&\overset{\eref{hi:1}}{\leq} K \max_{1\leq k \leq K}\PP \left\{ \frac{\| \VarX_n^\delta - \VarX_{n+k}^\delta \|^2}{\E\| \VarX_n^\delta - \VarX_{n+k}^\delta \|^2} > 4 \tau^2  \right\} \\
& \overset{\eref{ineq:gauss2}}{\leq} K \sqrt{2}e^{-\tau^2}.
\end{align*}
The second inequality follows directly, using that any $s(N)^{-\gamma} < \delta$ does not make any sense.
\end{proof}

\begin{lemma}\label{lemma:upperbound}
Assume that it holds $n_{opt} \geq n$, with $\omega_0$ big enough such that
\begin{equation}
\frac{c_1 \omega_0 \omega}{ 2\gamma -1}  \geq 1 \label{hi:3}
\end{equation}
and
\begin{equation}
\frac{ c_4 }{c_1} \leq 2 \label{hi:2}
\end{equation}
Then it holds that
\begin{equation*}
\PP \{ b_{N,\sigma}(n) < \tau \} \leq 32 e \omega^{K(2\lambda +2 \varepsilon + 1)} \tau^2
\omega^{-(n_{opt}-n)(2\lambda +2 \varepsilon + 2\gamma)}
\end{equation*}
where $\overline{\tau}$ is independent of $n$ and linearly dependent on $\tau$; $\overline{\omega} > \omega$ is independent of $n$ and linearly dependent on $\omega$. Furthermore, it holds
\begin{equation*}
\PP \{ B_{N,\sigma}(n) < \tau \} \leq 32 e \omega^{(2\lambda +2 \varepsilon + 1)} \tau^2
\omega^{-(n_{opt}-n)(2\lambda +2 \varepsilon + 2\gamma)}
\end{equation*}
\end{lemma}
\begin{proof}
It holds due to \eref{ineq1}, \eref{ineq2}, \eref{nopt} and \eref{ineq:gauss1}
\begin{align*}
\PP \{ b_{N,\sigma}(n) < \tau \}
&\leq  \PP \left\{ \forall_{1\leq k \leq K} :  4^{-1} \| \VarX_n^\delta - \VarX_{n+k}^\delta \| \varrho(n+k)^{-1} < \tau \right\}\\
&{\leq}  \min_{1\leq k \leq K}\PP \left\{ 4^{-1} \| \VarX_n^\delta - \VarX_{n+k}^\delta \| \varrho(n+k)^{-1} < \tau \right\}\\
&=  \min_{1\leq k \leq K}\PP \left\{ \frac{\| \VarX_n^\delta - \VarX_{n+k}^\delta \|^2}{\E\| \VarX_n^\delta - \VarX_{n+k}^\delta \|^2} < 16 \tau^2 \frac{\varrho(n+k)^2}{\E\| \VarX_n^\delta - \VarX_{n+k}^\delta \|^2} \right\}\\
&\overset{\eref{ineq1}\eref{ineq2}}{\leq}  \min_{1\leq k \leq K}\PP \left\{ \frac{\| \VarX_n^\delta - \VarX_{n+k}^\delta \|^2}{\E\| \VarX_n^\delta - \VarX_{n+k}^\delta \|^2} < 16 \tau^2 \frac{c_4 \frac{\delta^2 \omega_0^{2\lambda+2\varepsilon +1}}{2\lambda+2\varepsilon +1} \omega^{(n+k)(2\lambda+2\varepsilon +1)}}{c_1 \frac{\eta^2 \omega_0^{-2\gamma +1}}{ 2\gamma -1} \omega^{n(-2\gamma +1)}} \right\}\\
&\overset{\eref{nopt}}{\leq}  \min_{1\leq k \leq K}\PP \left\{ \frac{\| \VarX_n^\delta - \VarX_{n+k}^\delta \|^2}{\E\| \VarX_n^\delta - \VarX_{n+k}^\delta \|^2} <  16 \tau^2 \frac{c_4 \frac{\delta^2 \omega_0^{2\lambda+2\varepsilon +1}}{2\lambda+2\varepsilon +1} \omega^{(n+k)(2\lambda+2\varepsilon +1)}}{c_1 \frac{\delta^2 \omega_0^{2\lambda+2\varepsilon +1}}{2\lambda+2\varepsilon +1} \omega^{n_{opt}(2\lambda+2\varepsilon +1)} \omega^{(-2\gamma+1)(n-n_{opt})}} \right\}\\
& =  \min_{1\leq k \leq K}\PP \left\{ \frac{\| \VarX_n^\delta - \VarX_{n+k}^\delta \|^2}{\E\| \VarX_n^\delta - \VarX_{n+k}^\delta \|^2} < \tau^2 \frac{16 c_4 }{c_1} \omega^{-(n_{opt}-n) (2\lambda +2 \varepsilon + 2\gamma)} \omega^{k(2\lambda + 2\varepsilon + 1)}   \right\}\\
& \overset{\eref{hi:2}}{\leq}  \min_{1\leq k \leq K}\PP \left\{ \frac{\| \VarX_n^\delta - \VarX_{n+k}^\delta \|^2}{\E\| \VarX_n^\delta - \VarX_{n+k}^\delta \|^2} < 32 \tau^2 \omega^{-(n_{opt}-n) (2\lambda +2 \varepsilon + 2\gamma)} \omega^{k(2\lambda + 2\varepsilon + 1)}   \right\}\\
& \overset{\eref{ineq:gauss1}}{\leq} \min_{1\leq k \leq K}  \( 32 e \tau^2 \omega^{-(n_{opt}-n) (2\lambda +2 \varepsilon + 2\gamma)} \omega^{k(2\lambda + 2\varepsilon + 1)}\)^{\frac{c_1\frac{\eta^2 \omega_0^{-2\gamma +1}}{ 2\gamma -1} \omega^{n(-2\gamma +1)}}{ \eta^2\omega_0^{-2\gamma} \omega^{-2n\gamma}}} \\
& = \( 32 e \omega^{(2\lambda +2 \varepsilon + 1)} \tau^2  \omega^{-(n_{opt}-n) (2\lambda +2 \varepsilon +2\gamma)}\)^{{ \frac{c_1 \omega_0 \omega}{ 2\gamma -1}}} \\
& \overset{\eref{hi:3}}{\leq} 32 e \omega^{(2\lambda +2 \varepsilon + 1)} \tau^2
\omega^{-(n_{opt}-n)(2\lambda +2 \varepsilon + 2\gamma)}.
\end{align*}
The second inequality is trivial.
\end{proof}
This means that the balancing functional $b_{N,\sigma}$, respectively its smoothed version $B_{N,\sigma}$, shows the following behavior:
\begin{itemize}
\item Assume $n < n_{opt}$. The probability that $b(\cdot)$ falls below the threshold becomes smaller and smaller the farther away $n$ is from $n_{opt}$; near $n_{opt}$, one cannot make any sensible statements as in the above inequality the bound for the probability is bigger than $1$. In particular, the decay of probabilities is faster than the increase of error for smaller regularization parameters.
\item Besides the point $n_{opt}$, the probability of being above the threshold depends only on the level of the threshold.
\end{itemize}
Using this behavior, we can define the following method. This idea has already been presented in a different form in \cite{Raus/Haemarik:2008}, however in a purely deterministic setting with a focus on convergence results.
\begin{definition}[Fast Balancing]
Define
\begin{equation*}
n_{fb} = \argmin_{n } \{ b_{N,\sigma}(n) < \tau \}.
\end{equation*}
\end{definition}
\begin{theorem}\label{maintheorem}
Let $\sigma$ such that $K = 1$ and assume that $\omega_0$ is big enough such that \eref{hi:1},\eref{hi:3} and \eref{hi:2} hold; furthermore assume that $n_{opt}$ exists.

For any $N$ (including $N=\infty$) and any $\tau \geq 1$, $2\lambda + 2 \varepsilon >0$, the parameter $n_{fb}$ exists with probability $1$ and it holds
 the oracle inequality
\begin{equation*}
\E \| \VarX_{n_{fb}}^\delta - \VarX\|^2 \leq C \min_n \E \| \VarX_{n}^\delta -\VarX\|^2
\end{equation*}
where $C$ is not dependent on the particular $\VarX$ and $\xi$ (i.e., not on $\delta$ resp. $\eta$).
\end{theorem}
The proof we use is rather similar to the one used in \cite{Bauer/Reiss:2008}:
\begin{proof}
The proof consists of three parts:

Due to $K=1$, all random variables $b_{N,\sigma}(n) $ are independent. Hence, using lemma
\eref{lemma:lowerbound} it holds that
\begin{equation*}
\PP(n \geq n_{opt} + k ) \leq (\sqrt{2} e^{-1})^k \xrightarrow{k \rightarrow \infty} 0
\end{equation*}
as, due to the choice $K=1$, all random variables $b_{N,\sigma}(n)$ are independent. This trivially yields that $n_{fb}$ exists with probability $1$.

Hence we obtain using the H{\"o}lder inequality with $p^{-1} + \overline{p}^{-1}=1$
\begin{align*}
\E \| \VarX_{n_{fb}}^\delta - \VarX\|^2
=&   \sum_{n=0}^\infty \E \| \VarX - \VarX^\delta_n \|^2 \mathbf{1}_{n=n_{fb}} \\
\leq &   \sum_{n=0}^{n_{opt}-2} \(\E \| \VarX - \VarX^\delta_n \|^{2p}\)^{1/p} \(\E
\mathbf{1}_{n=n_{fb}}^{\overline{p}}\)^{1/\overline{p}}\\
&\qquad+ \max\left\{ \E \| \VarX - \VarX^\delta_{n_{opt}-1} \|^2 ,  \E \| \VarX - \VarX^\delta_{n_{opt}} \|^2\right\}\\
&\qquad+ \sum_{n=n_{opt}+1}^{\infty} \(\E \| \VarX - \VarX^\delta_n \|^{2p}\)^{1/p} \(\E
\mathbf{1}_{n=n_{fb}}^{\overline{p}}\)^{1/\overline{p}} \\
\end{align*}
In \cite{Bauer/Reiss:2008} it is proven using the Gaussian behavior that
\begin{equation}
\(\E \| \VarX - \VarX^\delta_n \|^{2p}\)^{1/p} \leq c_p \E \| \VarX - \VarX^\delta_n \|^{2} \label{hi:4}
\end{equation}
for some constant $c_p \geq 1$ depending only on $p$. Now using that $\lambda > -\varepsilon$ we can choose $\overline{p}$ near enough to $1$ such that
\begin{equation}\label{hi:n1}
2 \lambda + 2\varepsilon + 2\gamma (1-\overline{p}) + \overline{p}>0
\end{equation}
and furthermore assume that $\tau$ in relation to $\omega$ was chosen in such a way that
\begin{equation}\label{hi:n2}
\omega^{2\lambda+2\varepsilon +1} \(\sqrt{2}e^{-\tau^2}\)^{1/\overline{p}} < 1.
\end{equation}
Using lemmas
\ref{lemma:lowerbound} and \ref{lemma:upperbound}
\begin{align*}
\E \| \VarX_{n_{fb}}^\delta - \VarX\|^2
{\leq} & c_p \sum_{n=0}^{n_{opt}-2} \E \| \VarX - \VarX^\delta_n \|^{2} \(\PP \{ b_{N,\sigma}(n) < \tau \}\)^{1/\overline{p}} \\
&\qquad+  \omega^{(-2\gamma +1)} \E \| \VarX - \VarX^\delta_{n_{opt}} \|^2 \\
&\qquad+ c_p \sum_{n=n_{opt}+1}^{\infty} \(\E \| \VarX - \VarX^\delta_n \|^{2}\) \( \prod_{k=n_{opt}+1}^{n}  \PP \{ b_{N,\sigma}(k) > \tau \} \)^{1/\overline{p}}\\
\overset{\ref{lemma:lowerbound},\, \ref{lemma:upperbound},\, \eref{ineq3}}{\leq} & 2 c_p \E \| \VarX - \VarX^\delta_{n_{opt}} \|^2 \\&\qquad\( \sum_{n=0}^{n_{opt}-2} \(32 e \omega^{2\lambda +2 \varepsilon + 1} \tau^2  \omega^{-(n_{opt}-n)(2\lambda +2 \varepsilon + 2\gamma)}\)^{1/\overline{p}} c_6 c_5^{-1} \omega^{-(n_{opt}- n)(-2\gamma +1)}\right. \\
&\qquad\qquad +  \omega^{(-2\gamma +1)} \\
&\qquad\qquad + \left. \sum_{n=n_{opt}+1}^{\infty} \omega^{(n-n_{opt})(2\lambda+2\varepsilon +1)} \(\sqrt{2}e^{-\tau^2}\)^{(n-n_{opt})/\overline{p}} \) \\
\overset{\eref{hi:4}}{\leq} & \frac{C}{2} \E \| \VarX - \VarX^\delta_{n_{opt}} \|^2\\
\overset{\eref{ineq3}\eref{nopt}}{\leq} & C \min_n \E \| \VarX_{n}^\delta -\VarX\|^2
\end{align*}
due to the definition of $n_{opt}$ where
\begin{align*}
C  \overset{\eref{hi:n1},\eref{hi:n2}}{\leq} & 4 c_p \(  \( 32 e \omega^{2\lambda +2 \varepsilon + 1} \tau^2\)^{1/\overline{p}} c_6 c_5^{-1} (1 - \omega^{2\lambda + 2\varepsilon+2\gamma(1-\overline{p})+\overline{p}})^{-1/\overline{p}} \right. \\ 
& \left. + \omega^{(-2\gamma +1)} +  \( 1 - \omega^{2\lambda+2\varepsilon +1} \(\sqrt{2}e^{-\tau^2}\)^{1/\overline{p}} \)^{-1} \).
\end{align*}
Obviously $C$ is independent of the particular $x$ and $\xi$.
\end{proof}
This means in particular that we do \emph{not} lose a logarithmic factor and can set $\tau = 1$ without a problem as long as we keep $\omega$ small enough. Furthermore, this speeds up the method considerably since, as in the Morozov discrepancy principle, we no longer  need to find solutions for all $n$ up to $N$ but can stop after considering at most $n_* + K \approx n_{opt}+K$ solutions. Practice shows that the method works also for $K>1$ and even becomes more stable; however the proof would be unnecessarily complicated.

%% file: noisebehavior.tex
\section{Obtaining the Noise Behavior}

In practice, one often does not know $\varrho$ and therefore needs to estimate it. Nevertheless, in most practical situations it is possible to measure more than once or to partition the data into two or more data sets.

Assume that one can partition the measurement in two parts $y_1^{\tilde{\delta}}$ and
$y_2^{\tilde{\delta}}$ with $\tilde{\delta} = \sqrt{2} \delta$, we have
\begin{equation*}
\VarX_n^\delta = \frac{\VarX^{\tilde{\delta}}_{n,1} + \VarX^{\tilde{\delta}}_{n,2}}{2}
\end{equation*}
The estimate of $\varrho$ is now
\begin{equation*}
\widetilde{\varrho}(n) = \left\| \frac{\VarX^{\tilde{\delta}}_{n,1} -
\VarX^{\tilde{\delta}}_{n,2}}{2} \right\|
\end{equation*}
and it obviously holds
\begin{equation}\label{eq1}
\E \widetilde{\varrho}(n)^2 = \varrho(n)^2
\end{equation}
Accordingly, we can define $\tilde{b}_{N,\sigma}(n)$ by just replacing $\varrho$ with
$\tilde{\varrho}$.

This means that we can modify the probability estimations using a similar trick as in
\cite{Bauer/Reiss:2008}. It is important to notice that there is no way to reliably estimate the color of the noise based on only two solutions; the same holds for the noise level when the color of the noise is not known. Nevertheless, the information we obtain from two solutions is sufficient for optimal reconstructions.

\begin{lemma}
Assume that $n_{opt} < n$ and that $\omega_0$ is big enough such that
\begin{equation*}
\frac{c_3}{c_6}  \geq \frac{1}{2}.
\end{equation*}
and
\begin{equation}
\frac{c_3\omega_0 }{1+2\lambda+2\varepsilon}>1 \label{hi:5}
\end{equation}
Then it holds if $\frac{e}{2\tau} < 1$ that
\begin{equation*}
\PP \{ \widetilde{b}_{N,\sigma}(n) > \tau \} \leq K  \frac{e}{\tau}
\end{equation*}
\end{lemma}
\begin{proof}
Using \eref{eq1}, \eref{ineq2}, lemma \ref{ProbLemma} and parts which have already been shown in \ref{lemma:lowerbound}, it holds:
\begin{align*}
K^{-1}\PP \{ \widetilde{b}_{N,\sigma}(n) > \tau \}
&\leq K^{-1}\sum_{1\leq k \leq K}\PP \left\{ 4^{-1} \| \VarX_n^\delta - \VarX_{n+k}^\delta \| \widetilde{\varrho}(n+k)^{-1} > \tau \right\}\\
&\leq  \max_{1\leq k \leq K}\PP \left\{ 4^{-1} \| \VarX_n^\delta - \VarX_{n+k}^\delta \| \widetilde{\varrho}(n+k)^{-1} > \tau \right\}\\
&\overset{\eref{eq1}}{=}  \max_{1\leq k \leq K}\PP \left\{ \frac{\| \VarX_n^\delta - \VarX_{n+k}^\delta \|^2}{\E\| \VarX_n^\delta - \VarX_{n+k}^\delta \|^2} > 16 \tau^2 \frac{{\varrho}(n+k)^2}{\E\| \VarX_n^\delta - \VarX_{n+k}^\delta \|^2}\frac{\widetilde{\varrho}(n+k)^2}{\varrho(n+k)^2} \right\}\\
&\overset{\ref{lemma:lowerbound}}{\leq} \max_{1\leq k \leq K}\PP \left\{ \frac{\| \VarX_n^\delta - \VarX_{n+k}^\delta \|^2}{\E\| \VarX_n^\delta - \VarX_{n+k}^\delta \|^2} > 4 \tau^2 \frac{\widetilde{\varrho}(n+k)^2}{\E \widetilde{\varrho}(n+k)^2} \right\} \\
& \leq \max_{1\leq k \leq K}\PP \left\{ \frac{\| \VarX_n^\delta - \VarX_{n+k}^\delta \|^2}{\E\| \VarX_n^\delta - \VarX_{n+k}^\delta \|^2} > 2 \tau  \right\} + \PP \left\{ \frac{1}{2\tau} > \frac{\widetilde{\varrho}(n+k)^2}{\E \widetilde{\varrho}(n+k)^2} \right\}\\
& \overset{\eref{ineq2}\eref{ineq:gauss1}\eref{ineq:gauss2}}{\leq} \sqrt{2} e^{-\tau/2} + \max_{1\leq k \leq K} \(\frac{e}{2\tau}\)^{\frac{c_3\frac{\delta^2 \omega_0^{1+2\lambda+2\varepsilon}}{1+2\lambda+2\varepsilon}\omega^{(n+k)(2\lambda+2\varepsilon+1)}}{\delta^2 \omega_0^{2\lambda+2\varepsilon}\omega^{(n+k)(2\lambda+2\varepsilon)}}}\\
& \leq \sqrt{2} e^{-\tau/2} + \(\frac{e}{2\tau}\)^{c_3\frac{\omega_0 \omega^{n+k}}{1+2\lambda+2\varepsilon}} \\
& \overset{\eref{hi:5}}{\leq} \sqrt{2} e^{-\tau/2} + \frac{e}{2\tau} \leq \frac{e}{\tau}.
\end{align*}
\end{proof}

\begin{lemma}
Assume that it holds $n_{opt} \geq n$ and assume that $\omega$ is big enough such that
\begin{equation*}
\frac{c_1 \omega_0 \omega}{ 2\gamma -1}  \geq 1
\end{equation*}
and
\begin{equation*}
\frac{ c_4 }{c_1} \leq 2
\end{equation*}
Then it holds that
\begin{equation*}
\PP \{ \widetilde{b}_{N,\sigma}(n) < \tau \} \leq 64 e \omega^{2\lambda +2 \varepsilon +
1} \tau  \omega^{-(n_{opt}-n)(\lambda +2 \varepsilon + 2\gamma)}
\end{equation*}
\end{lemma}
\begin{proof}
It holds using lemma \ref{ProbLemma} and parts of lemma \ref{lemma:upperbound}:
\begin{align*}
\PP \{ \widetilde{b}_{N,\sigma}(n) < \tau \}
&\leq  \PP \left\{ \forall_{1\leq k \leq K} :  4^{-1} \| \VarX_n^\delta - \VarX_{n+k}^\delta \| \widetilde{\varrho}(n+k)^{-1} < \tau \right\}\\
&\leq  \min_{1\leq k \leq K}\PP \left\{ 4^{-1} \| \VarX_n^\delta - \VarX_{n+k}^\delta \| \widetilde{\varrho}(n+k)^{-1} < \tau \right\}\\
&\overset{\eref{eq1}}{=} \min_{1\leq k \leq K}\PP \left\{ \frac{\| \VarX_n^\delta - \VarX_{n+k}^\delta \|^2}{\E\| \VarX_n^\delta - \VarX_{n+k}^\delta \|^2} < 16 \tau^2 \frac{{\varrho}(n+k)^2}{\E\| \VarX_n^\delta - \VarX_{n+k}^\delta \|^2} \frac{\widetilde{\varrho}(n+k)^2}{\varrho(n+k)^2} \right\}\\
& \overset{\ref{lemma:upperbound}}{\leq}  \PP \left\{ \frac{\| \VarX_n^\delta - \VarX_{n+1}^\delta \|^2}{\E\| \VarX_n^\delta - \VarX_{n+1}^\delta \|^2} < 32 \tau^2 \omega^{-(n_{opt}-n) (2\lambda +2 \varepsilon + 2\gamma)} \omega^{k(2\lambda + 2\varepsilon + 1)} \frac{\widetilde{\varrho}(n+1)^2}{\E \widetilde{\varrho}(n+1)^2}  \right\}\\
& \leq  \PP \left\{ \frac{\| \VarX_n^\delta - \VarX_{n+1}^\delta \|^2}{\E\| \VarX_n^\delta - \VarX_{n+1}^\delta \|^2} < 32 \tau \omega^{-(n_{opt}-n) (\lambda +2 \varepsilon + 2\gamma)}  \omega^{k(2\lambda + 2\varepsilon + 1)} \right\} \\&\qquad\qquad+ \PP \left\{ \tau \omega^{(n_{opt}-n) \lambda }< \frac{\widetilde{\varrho}(n+1)^2}{\E \widetilde{\varrho}(n+1)^2} \right\}\\
& \overset{\ref{lemma:upperbound}, \eref{ineq:gauss2}}{\leq} 32 e \omega^{(2\lambda +2 \varepsilon + 1)} \tau
\omega^{-(n_{opt}-n)(\lambda +2 \varepsilon + 2\gamma)} + \sqrt{2} e^{-\tau \omega^{(n_{opt}-n) \lambda }/4} \\
& \leq 64 e \omega^{2\lambda +2 \varepsilon + 1} \tau
\omega^{-(n_{opt}-n)(\lambda +2 \varepsilon + 2\gamma)}
\end{align*}
\end{proof}
This means that, in principle, the balancing functional $\widetilde{b}_{N,\sigma}$ shows the same behavior as its non-estimated counterpart ${b}_{N,\sigma}$.

Using this behavior, we can define a version of the new method:
\begin{definition}[Fast Balancing]
Define
\begin{equation*}
n_{fb} = \argmin_{n } \{ \widetilde{b}_{N,\sigma}(n) < \tau \}.
\end{equation*}
\end{definition}

\begin{theorem}
Let $\sigma$ such that $K = 1$ and assume that $\omega_0$ is big enough such that \eref{hi:1}, \eref{hi:3}, \eref{hi:2} and \eref{hi:5} hold; furthermore assume that $n_{opt}$ exists.

For any $N$ (including $N=\infty$) and any $\tau \geq 1$, $\lambda + 2 \varepsilon >0$, the parameter $n_{fb}$ exists with probability $1$ and it holds
 the oracle inequality
\begin{equation*}
\E \| \VarX_{n_{fb}}^\delta - \VarX\|^2 \leq C \min_n \E \| \VarX_{n}^\delta -\VarX\|^2
\end{equation*}
where $C$ is not dependent on the particular $\VarX$ and $\xi$ (i.e., not on $\delta$
resp. $\eta$).
\end{theorem}
The proof works in the exact same way as for theorem \ref{maintheorem}.

%% file: conclusion.tex
\section{Conclusion}

Assuming that our model is suitable for describing real data, we have presented an answer to the initial questions, at least for the newly defined methods:
\begin{itemize}
\item We do not lose a logarithmic factor, because the probability of the balancing principle going completely wrong is negligibly small. 
\item We do not need explicit knowledge of the noise level $\delta$ and the noise behavior. A rough estimation based on two independent measurements is sufficient.
\item The newly introduced method is  as fast as the Morozov discrepancy principle (if one neglects constant factors).
\end{itemize}
Although the situation is not completely comparable with the case of deterministic $x$ which suffers from the mentioned logarithmic factor we think this is a significant advance to understand the difference in theoretical and actual behavior of the balancing principle.
Though it has not been shown in this paper, one can transfer parts of the proofs also to the case of Tikhonov regularization \cite{Bauer:2009}.

Furthermore, large numerical experiments show that the newly defined method works very well and can, in contrast to most other parameter choice regimes, cope with colored noise without any performance loss \cite{Bauer/Lukas:2010}. In these experiments it was observed that the factor $C$ in the oracle inequality is at most around $2$. The method is very stable, i.e., the number of observed outliers is very low, both for Tikhonov and Spectral-Cut-Off regularization.

Additionally it was observed that the stability increases if one uses more than two measurements in order to estimate the noise behavior and if one chooses $K$ a bit bigger than $1$.

\section*{Acknowledgements}

The author gratefully acknowledges the financial support by the Upper
Austrian Technology and Research Promotion.